\documentclass[reqno]{amsart}
\usepackage[foot]{amsaddr}

\usepackage{blkarray}
\usepackage[hidelinks]{hyperref} 
\usepackage{url}
\usepackage{amssymb}
\usepackage{amscd}
\usepackage{amsthm} 
\usepackage{setspace}
\usepackage{enumerate}
\usepackage{graphicx}
\usepackage{mathtools}
\usepackage{subfigure}
\usepackage{siunitx}
\usepackage{tikz-cd}
\usepackage{bm}
\numberwithin{equation}{section}
\usepackage[obeyFinal,textsize=footnotesize]{todonotes}

\usepackage{tcolorbox}
\usepackage{float}
\usepackage{subcaption}

\usepackage{mathtools}%                  http://www.ctan.org/pkg/mathtools
\usepackage[tableposition=top]{caption}% http://www.ctan.org/pkg/caption
\usepackage{booktabs,dcolumn}%           http://www.ctan.org/pkg/dcolumn + http://www.ctan.org/pkg/booktabs

\title[Perturbation analysis of Higham squared maximum growth matrices]{On a perturbation analysis of Higham squared maximum Gaussian elimination growth matrices}
\author{Alan Edelman}
\address{\scriptsize Department of Mathematics, Massachusetts Institute of Technology, Cambridge, MA, 02139 USA.
}
\email{edelman@mit.edu}
\author{John Urschel}
\email{urschel@mit.edu}
\author{Bowen Zhu}
\email{bowenzhu@g.harvard.edu}
\address{\scriptsize School of Engineering and Applied Sciences, Harvard University, Cambridge, MA, 02138 USA.
}
\subjclass[2020]{Primary 65F05, 15A23. \\ \indent This arXiv technical report is an abbreviated complement to the third author's Harvard master's thesis \cite{bowen24}.}
\newtheorem{theorem}{Theorem}[section]
\newtheorem{lemma}[theorem]{Lemma}

\newtheorem{proposition}[theorem]{Proposition}

\newtheorem{corollary}[theorem]{Corollary}

\begin{document}

\begin{abstract}
%Gaussian elimination is the most popular technique for solving a dense linear system. Large errors in this procedure can occur in floating point arithmetic when the matrix's growth factor is large. We study this potential issue and how perturbations can improve the robustness of the Gaussian elimination algorithm. In their 1989 paper, Higham and Higham characterized the complete set of real n by n matrices that achieves the maximum growth factor under partial pivoting. This set of matrices serves as the critical focus of this work. Through theoretical insights and empirical results, we illustrate the high sensitivity of the growth factor of these matrices to perturbations and show how subtle changes can be strategically applied to matrix entries to significantly reduce the growth.

Gaussian elimination is the most popular technique for solving a dense linear system. Large errors in this procedure can occur in floating point arithmetic when the matrix's growth factor is large. In the study of numerical linear algebra, it is often valuable to study and characterize the worst case examples. To this end, in their 1989 paper, Higham and Higham characterized the complete set of real n by n matrices that achieves the maximum growth factor under partial pivoting. Left undone is a sensitivity analysis for these matrices under perturbations. The growth factor of these and nearby matrices is the subject of this work. Through theoretical insights and empirical results, we illustrate the high sensitivity of the growth factor of these matrices to perturbations and show how subtle changes can be strategically applied to matrix entries to significantly reduce the growth.
\end{abstract}

\maketitle

\section{Introduction to Partial Pivoting Growth}

It is well-known, even to beginning students in engineering and science, that
the most popular method for solving the dense linear system $Ax=b$ for $x$ 
is Gaussian elimination. Indeed, this method, employed using partial pivoting, is widely
available through interfaces from high level languages such as
Julia \cite{bezanson2017julia}, Mathematica \cite{Mathematica}, Matlab \cite{MATLAB}, Python NumPy \cite{harris2020array}, R
\cite{rmanual}, etc. Error estimates for the stability of the Gaussian elimination algorithm in floating point arithmetic are governed by the bits of precision used, the condition number of the matrix, and the growth factor (i.e., the largest magnitude entry encountered during the Gaussian elimination algorithm) \cite[Theorem 3.3.2]{golub2013matrix}. Many researchers have studied and continue to study the question of why Gaussian elimination with partial pivoting has been so very effective \cite{foster1994gaussian,higham2021random,higham1989large,huang2022average,sankar2006smoothed,spielman2004smoothed,trefethen2022numerical,trefethen1990average}. In contrast to complete pivoting, where the existence of matrices with even super-linear growth remains an open problem \cite{bisain2023new,edelman2024some,gould1991growth}, it has been known since Wilkinson's classic text {\it The Algebraic Eigenvalue Problem} \cite[p.212]{Wilkinson1965AEP} that, for partial pivoting, the growth factor is bounded above by $2^{n-1}$ and that this quantity can be achieved by the matrix
\begin{equation}\label{eqn:wilk}
A = \begin{pmatrix*}[r] 1 & 0\phantom{.} & \cdots & 0 & 1 \\ -1 & \ddots & \ddots & \vdots & \vdots \\
\vdots & \ddots & 1 &  0 & 1 \\
-1 & \cdots & -1 & 1 & 1 \\
-1 & \cdots & -1 & -1 & 1 \end{pmatrix*}.
\end{equation}
Much later, Higham and Higham (from now on denoted  {Higham}\footnote[2] 
{We denote ``Higham and Higham" as Higham$^2$, to be read as ``Higham squared," yet we realize
this has the appearance of a footnote, so for readers who saw it this way, we have included this footnote.
}) identified the complete set of $n$ by $n$ real matrices that achieve the maximal growth of $2^{n-1}$  \cite{higham1989large}. We call such matrices Higham$^2$ matrices (see Proposition \ref{prop} for a description). A scalar quantity of interest is the last pivot of a Higham$^2$ matrix, which is a differentiable function of the matrix entries.
We can therefore ask for the gradient of this last pivot or, even better, to have a full (non-infinitesimal) perturbation analysis of the last pivot (for Gaussian elimination without pivoting).  We provide such a perturbation analysis in Theorem \ref{theorem}. The last pivot is an ideal quantity to measure in order to understand the growth factor, as every entry of $U$ is the last pivot of the LU factorization of some submatrix of $A$. We observe that generically, large growth does not last very long in the sense that
often a small perturbation can dramatically reduce a large pivot. We have a mental image that the Higham$^2$ matrices 
live on a kind of ``ridge" that one can easily fall off of. This picture is consistent with the smoothed analysis of Sankar, Spielman, and Teng \cite{sankar2006smoothed}. The structure of the Higham$^2$ matrices provide an ideal setting to better understand the ridge and its profile. Perhaps unsurprisingly, not all directions of descent are created equal. We provide numerical experiments (in Section \ref{sec:experiments}) to visualize the effects of perturbing Higham$^2$ matrices and confirm the conclusions gleaned from the theoretical results of Section \ref{sec:theory}.

\section{Entrywise Perturbations \& Higham$^2$ Matrices}\label{sec:theory}

Here we provide mathematical estimates for the effects of entrywise perturbations on the last pivot of the LU factorization (without pivoting) of a matrix (Lemma \ref{lemma}), recall a characterization of Higham$^2$ matrices (Proposition \ref{prop}), and consider the effects of entrywise perturbations on this class (Theorem \ref{theorem}, Corollaries \ref{cor:cond} \& \ref{cor:topright}). These theoretical results give insight into the experimental observations in Section \ref{sec:experiments}.

\begin{lemma} \label{lemma}
Let
$$ A = \begin{pmatrix} \hat L & 0 \\  \bm{\ell}^T & 1 \end{pmatrix} \begin{pmatrix} \hat U & \bm{u} \\ 0 & p \end{pmatrix} \in \mathrm{GL}_n(\mathbb{R}),$$
where $\hat L \in \mathrm{SL}_{n-1}(\mathbb{R})$ is lower unitriangular, $\hat U \in \mathrm{GL}_{n-1}(\mathbb{R})$ is upper triangular, and $\bm{\ell},\bm{u} \in \mathbb{R}^{n-1}$. Then the LU factorization, if it exists, of $A + \epsilon \,  \bm{e}_i \bm{e}_j^T$, where $\bm{e}_i$ is the $i^{th}$ standard basis vector, has last pivot
\[ p^{(i,j)}_\epsilon = p + \epsilon \frac{(\hat U^{-1} \bm{u})_j(\bm{\ell}^T \hat L^{-1})_i}{1+ \epsilon (\hat L\hat U)^{-1}_{ji}} \qquad \text{for} \quad i,j < n,\]
$p^{(i,n)}_\epsilon = p  - \epsilon (\bm{\ell}^T \hat L^{-1})_i$ for $i<n$, $p^{(n,j)}_\epsilon = p - \epsilon (\hat U^{-1} \bm{u})_j$ for $j<n$, and $p^{(n,n)}_\epsilon = p+\epsilon$.
\end{lemma}

\begin{proof}
We first consider the case $i,j <n$. The last pivot is a ratio of determinants
$$p^{(i,j)}_\epsilon = \frac{\det (A+\epsilon \,  \bm{e}_i \bm{e}_j^T) }{\det (\hat L\hat U + \epsilon \,\bm{\hat e}_i  \bm{\hat e}_j^T )} = \frac{\det(A)}{\det(\hat L\hat U)} \, \bigg( \frac{1+\epsilon \bm{e}_j^T A^{-1} \bm{e}_i}{1+ \epsilon \bm{\hat e}_j^T (\hat L\hat U)^{-1} \bm{\hat e}_i}\bigg) = p \bigg( \frac{1+\epsilon A^{-1}_{ji}}{1+ \epsilon (\hat L\hat U)^{-1}_{ji}} \bigg),$$
where $\bm{\hat e}_i$ is the $i^{th}$ standard basis vector in $\mathbb{R}^{n-1}$. The matrix $A^{-1}$ has block form
$$ A^{-1} = \begin{pmatrix} \hat U^{-1} & -p^{-1} \hat U^{-1}\bm{u} \\ 0 & p^{-1} \end{pmatrix} \begin{pmatrix} \hat L^{-1} & 0 \\  -\bm{\ell}^T \hat L^{-1} & 1 \end{pmatrix} = \begin{pmatrix} (\hat L\hat U)^{-1} + p^{-1} \hat U^{-1} \bm{u} \bm{\ell}^T \hat L^{-1} & - p^{-1} \hat U^{-1} \bm{u} \\ -p^{-1} \bm{\ell}^T \hat L^{-1}  & p^{-1} \end{pmatrix},$$
and so $A^{-1}_{ji} = (\hat L\hat U)_{ji}^{-1}+p^{-1} (\hat U^{-1} \bm{u})_j ( \bm{\ell}^T \hat L^{-1})_i$. Therefore, 
$$ p^{(i,j)}_\epsilon - p = p \bigg[  \frac{1+\epsilon A^{-1}_{ji}}{1+ \epsilon (\hat L\hat U)^{-1}_{ji}} - 1 \bigg] = \frac{\epsilon p(A^{-1}_{ji}-(\hat L\hat U)^{-1}_{ji})}{1+ \epsilon (\hat L\hat U)^{-1}_{ji}} = \frac{\epsilon (\hat U^{-1} \bm{u})_j ( \bm{\ell}^T \hat L^{-1})_i } {1+ \epsilon (\hat L\hat U)^{-1}_{ji}}.$$
When $i = n$ or $j = n$, $\displaystyle{p^{(i,j)}_\epsilon =  \frac{\det (A+\epsilon \,  \bm{e}_i \bm{e}_j^T) }{\det (\hat L\hat U)} = p (1+\epsilon A^{-1}_{ji}).}$
Noting that $A_{jn}^{-1} = - p^{-1} (\hat U^{-1} \bm{u})_j$ for $j <n$, $A_{ni}^{-1} = - p^{-1} (\bm{\ell}^T \hat L^{-1})_i$ for $i <n$, and $A_{nn}^{-1} = p^{-1}$ completes the proof.
\end{proof}

We recall the following characterization of Higham$^2$ matrices from the original paper of Higham and Higham \cite{higham1989large}, where we have slightly adjusted the normalization and notation to suit our needs.

\begin{proposition}{\cite[Theorem 2.2]{higham1989large}}\label{prop}
Every matrix $A \in \mathrm{GL}_n(\mathbb{R})$, $\|A\|_{\max} = 1$, with growth factor under partial pivoting equal to $2^{n-1}$ is of the form
\begin{equation}\label{eq:higham_desc} 
D P A =  \begin{pmatrix} \hat L & 0 \\  - \bm{1}^T & 1 \end{pmatrix} \begin{pmatrix} \hat U & \bm{u} \\ 0 & 2^{n-1} \end{pmatrix} = \begin{pmatrix} \hat L\hat U & \bm{1} \\ - \bm{1}^T \hat U & 1 \end{pmatrix},
\end{equation}
where $D\in \mathrm{GL}_n(\mathbb{R})$ is a $\pm1$ diagonal matrix, $P \in \mathrm{GL}_n(\mathbb{R})$ is a permutation matrix associated with a partial pivoting of $A$, $\hat L \in \mathrm{GL}_{n-1}(\mathbb{R})$ is lower unitriangular with $\hat L_{ij} = -1$ for all $i>j$, $\bm{1}$, $\bm{u} = (1,2,...,2^{n-2})^T \in \mathbb{R}^{n-1}$, and $\hat U \in \mathrm{GL}_{n-1}(\mathbb{R})$ is upper triangular, with entries satisfying $\|\hat L\hat U\|_{\max} \le 1$ and $\|\bm{1}^T \hat U\|_\infty \le 1$. 
\end{proposition}

Applying Lemma \ref{lemma} to Higham$^{2}$ matrices (described in Proposition \ref{prop}) gives the following theorem.

\begin{theorem}\label{theorem}
Let $A \in \mathrm{GL}_n(\mathbb{R})$, $\|A\|_{\max} = 1$, be a Higham$^2$ matrix of the form in Equation \ref{eq:higham_desc} with $P = D = I$. Then the LU factorization, if it exists, of $A + \epsilon \,  \bm{e}_i \bm{e}_j^T$ has last pivot
\begin{equation}\label{eq:higham}
p_{\epsilon}^{(i,j)} =  \begin{cases} \displaystyle{\frac{1}{2^{1-n} +\epsilon \, 2^{-i} \sum_{\ell=1}^{n-j} 2^{-\ell} \hat U_{j,n-\ell}^{-1}}} & \qquad \text{for} \quad i <j < n \\[3 ex]
 \displaystyle{\frac{1 + \tfrac{1}{2} \epsilon \, ( \hat U_{ji}^{-1} - \sum_{\ell = 1}^{i-j} 2^{-\ell} \hat U_{j,i-\ell}^{-1} )}{2^{1-n} + \epsilon (2^{1-n} \hat U_{ji}^{-1} +2^{-i} \sum_{\ell = 1}^{n-i-1} 2^{-\ell} \hat U_{j,n-\ell}^{-1} )}} & \qquad \text{for} \quad j \le i < n \\[1 ex]
\end{cases}\quad ,
\end{equation}
$p_{\epsilon}^{(i,n)} = 2^{n-1}(1+\epsilon \, 2^{-i})$ for $i <n$, $p_{\epsilon}^{(n,j)} = 2^{n-1}(1 - \epsilon \sum_{\ell = 1}^{n-j} 2^{-\ell} \hat U_{j,n-\ell}^{-1})$ for $j < n$, and $p_{\epsilon}^{(n,n)}= 2^{n-1} + \epsilon$.
\end{theorem}

\begin{proof}
By Proposition \ref{prop}, $A = \begin{pmatrix} \hat L & 0 \\  - \bm{1}^T & 1 \end{pmatrix} \begin{pmatrix} \hat U & \bm{u} \\ 0 & 2^{n-1} \end{pmatrix}$, where $\hat L \in \mathrm{SL}_{n-1}(\mathbb{R})$ is lower unitriangular with $\hat L_{ij} = -1$ for all $i>j$, $\bm{u} = (1,2,...,2^{n-2})^T$, and $\hat U \in \mathrm{GL}_{n-1}(\mathbb{R})$ is upper triangular. The matrix $\hat L$ has a simple structure, and its inverse has entries given by $\hat L^{-1}_{ij} = \phi(i-j)$, where $\phi(k)$ equals zero for $k<0$, one for $k = 0$, and $2^{k-1}$ for $k>0$. Therefore,
\[(\bm{1}^T \hat L^{-1})_i = \sum_{k = i}^{n-1} \phi(k-i) = 2^{n-i-1} \qquad \text{and} \qquad  (\hat L\hat U)_{ji}^{-1} = \sum_{k = i}^{n-1} \phi(k-i) \hat U_{jk}^{-1}=
\hat U_{ji}^{-1} + \sum_{k = i+1}^{n-1} 2^{k-i-1} \hat U_{jk}^{-1} . \]
Applying Lemma \ref{lemma} to $A$ and noting that
$(\hat U^{-1} \bm{u})_j = \sum_{k = j}^{n-1} 2^{k-1} \hat U^{-1}_{jk}$, we have $p_{\epsilon}^{(n,n)} = 2^{n-1} + \epsilon$, $p_{\epsilon}^{(i,n)} = 2^{n-1} (1+ \epsilon \, 2^{-i})$ for $i<n$, $p_{\epsilon}^{(n,j)} = 2^{n-1}(1 - \epsilon \sum_{\ell = 1}^{n-j} 2^{-\ell} \hat U_{j,n-\ell}^{-1})$ for $j<n$, and 
\[p_\epsilon^{(i,j)} = 2^{n-1} - \frac{ \epsilon 2^{n-i-1} \sum_{k = j}^{n-1} 2^{k-1} \hat U^{-1}_{jk}}{1+\epsilon(\hat U_{ji}^{-1} + \sum_{k = i+1}^{n-1} 2^{k-i-1} \hat U_{jk}^{-1} )}\qquad \text{for} \quad i,j < n. \]
When $i<j<n$, $\hat U_{ji}^{-1} = 0$ and $\sum_{k = i+1}^{n-1} 2^{k-i-1} \hat U_{jk}^{-1} = \sum_{k = j}^{n-1} 2^{k-i-1} \hat U^{-1}_{jk}$, and so $p_{\epsilon}^{(i,j)}$ equals $(2^{1-n} +\epsilon \, 2^{-i} \sum_{\ell=1}^{n-j} 2^{-\ell} \hat U_{j,n-\ell}^{-1})^{-1}$. Finally, in the case $j \le i<n$, we have
\[p_\epsilon^{(i,j)} = \frac{1 + \tfrac{1}{2} \epsilon \, ( \hat U_{ji}^{-1} - \sum_{\ell = 1}^{i-j} 2^{-\ell} \hat U_{j,i-\ell}^{-1} )}{2^{1-n} + \epsilon (2^{1-n} \hat U_{ji}^{-1} +2^{-i} \sum_{\ell = 1}^{n-i-1} 2^{-\ell} \hat U_{j,n-\ell}^{-1} )}.\]
\end{proof}

Equation \ref{eq:higham} of Theorem \ref{theorem} deserves a number of observations. First, we note a connection between the rate at which $p_\epsilon^{(i,j)}$ decreases and the condition number of the matrix: if the last pivot under $\epsilon$ perturbation is much smaller than $1/|\epsilon|$, then $A$ is ill-conditioned.

\begin{corollary}\label{cor:cond}
Let $A \in \mathrm{GL}_n(\mathbb{R})$, $\|A\|_{\max} = 1$, be a Higham$^2$ matrix of the form in Equation \ref{eq:higham_desc} with $P = D = I$. If $|\epsilon|<1$ and $\sqrt{n}<|p_{\epsilon}^{(i,j)}|<2^{n-5}$, then $\kappa_2(A)= \|A\|_2 \|A^{-1}\|_2 \ge \sqrt{n} \, \big| 3 \epsilon p_\epsilon^{(i,j)} \big|^{-1}$.
\end{corollary}

\begin{proof}
Let $|\epsilon|<1$ and $\sqrt{n}<|p_{\epsilon}^{(i,j)}|<2^{n-5}$ for some $i,j$. We first relate $\kappa_2(A)$ to the entries of $U^{-1}$:
\[ \|A\|_2 \|A^{-1}\|_2 \ge \|A \mathbf{e}_n\|_2 \|A^{-1}\|_2 = \sqrt{n}\max_{ y \ne 0} \frac{\|U^{-1} y\|_2}{\|L y\|_2} \ge \sqrt{n}\max_{\ell} \frac{\|U^{-1} e_{n-\ell}\|_2}{\|L e_{n-\ell}\|_2} \ge \max_{k,\ell} \frac{\sqrt{n} \, |U^{-1}_{k,n-\ell}|}{\sqrt{\ell+1}} .\]
What remains is to show that $|U^{-1}_{k,n-\ell}| \ge \sqrt{\ell+1} \,  \big|3 \epsilon p_{\epsilon}^{(i,j)} \big|^{-1}$ for some $k,\ell$. We proceed by contradiction. If this is not the case, then
\[\left| \epsilon \sum_{\ell = 1}^{k} \frac{\hat U_{j,n-\ell}^{-1}}{2^{\ell} } \right| < \frac{1}{\big|p_{\epsilon}^{(i,j)} \big|} \sum_{\ell = 1}^{\infty}  \frac{\sqrt{\ell+1}}{3 \times 2^{\ell}} < \frac{0.57}{\big|p_{\epsilon}^{(i,j)} \big|} \quad \text{and} \quad \left| \epsilon \, ( \hat U_{ji}^{-1} - \sum_{\ell = 1}^{i-j} 2^{-\ell} \hat U_{j,i-\ell}^{-1} ) \right| < \sum_{\ell = 0}^\infty \frac{1}{3 \times 2^{\ell}} = \frac{2}{3}. \]
These two bounds, combined with the formulae of Theorem \ref{theorem}, lead to a contradiction for all choices of $(i,j)$. For example, if $j \le i<n$, then 
\[\big|p_\epsilon^{(i,j)}\big| \ge \frac{1 - \tfrac{1}{2} \left|\epsilon \, ( \hat U_{ji}^{-1} - \sum_{\ell = 1}^{i-j} 2^{-\ell} \hat U_{j,i-\ell}^{-1} )\right|}{2^{1-n} + \left|\epsilon (2^{1-n} \hat U_{ji}^{-1} +2^{-i} \sum_{\ell = 1}^{n-i-1} 2^{-\ell} \hat U_{j,n-\ell}^{-1} )\right|} > \frac{2/3}{2^{1-n} + 0.57/\big|p_{\epsilon}^{(i,j)}\big|}> \frac{2/3}{2^{-4}+0.57}  \, \big|p_{\epsilon}^{(i,j)}\big|,\]
a contradiction. The remaining cases are similar, and are left to the reader.
\end{proof}

Now, let us restrict our attention to the case $i<j<n$, with $i$ relatively small. As long as $\epsilon \, 2^{-i} \sum_{\ell=1}^{n-j} 2^{-\ell} \hat U_{j,n-\ell}^{-1}$ is not exponentially small in $n$, $p_\epsilon^{(i,j)}$ is approximately $2^{i} \, (\epsilon \, \sum_{\ell=1}^{n-j} 2^{-\ell} \hat U_{j,n-\ell}^{-1})^{-1}$. One would expect this to  be the case for ``most" Higham$^2$ matrices when $\epsilon$ is only polynomially small, though exceptions certainly exist (e.g., the Wilkinson matrix $\hat U = I$). This intuition is supported by experimental results in Section \ref{sec:experiments}. The case $i = 1$ and $j = n-1$ is particularly striking, and gives guaranteed improvement for all Higham$^2$ matrices.
\begin{corollary}\label{cor:topright}
Let $A \in \mathrm{GL}_n(\mathbb{R})$, $\|A\|_{\max} = 1$, be a Higham$^2$ matrix of the form in Equation \ref{eq:higham_desc} with $P = D = I$. Then
\begin{equation}\label{ineq:top_right} | p^{(1,n-1)}_\epsilon| \le \frac{4 +o_n(1)}{|\epsilon|}, \qquad \text{where} \qquad |o_n(1)| < 2^{-(n-6)} |\epsilon|^{-1} \; \text{for} \; |\epsilon| > 2^{-(n-4)} .
\end{equation}
\end{corollary}

\begin{proof}
By Equation \ref{eq:higham},
$$ p^{(1,n-1)}_\epsilon = \frac{4 \hat U_{n-1,n-1}}{\epsilon + 2^{-(n-3)} \hat U_{n-1,n-1} } = \frac{\hat U_{n-1,n-1}}{\epsilon}\bigg( 4 - \frac{2^{-(n-5)} \hat U_{n-1,n-1}}{\epsilon + 2^{-(n-3)} \hat U_{n-1,n-1}} \bigg).$$
Noting that $|\hat U_{n-1,n-1}| = \big| \big[ (\hat L\hat U)_{n-1,n-1} +(\bm{1}^T \hat U)_{n-1} \big] / 2 \big| \le 1$ gives the desired result.
\end{proof}

The entry $(1,n-1)$ is an example of a perturbation direction that always produces a small last pivot when $\epsilon$ is only polynomially small. Finally, we note that, while Lemma \ref{lemma} and Theorem \ref{theorem} apply only to the last pivot, this general framework holds for arbitrary entries of $U$, as every entry of $U$ is the last pivot of the LU factorization of some submatrix of $A$. For instance, Corollary \ref{cor:topright} implies that a only polynomially small $\epsilon$ perturbation to the $(1,n-k)^{th}$ entry of a Higham$^2$ matrix leads to $|U_{n-k+1,n}| \le (4 + o_{n}(1))/|\epsilon|$ for any fixed $k$ (with $n$ growing). In Section \ref{sec:experiments}, we make use of the insights gained from Theorem \ref{theorem} to suggest perturbations tailored to the most influential components of $A$, and compare their effect to perturbations applied uniformly to $A$.

\section{Experimental Results}\label{sec:experiments}

Here we perform two experiments, illustrated in Figures \ref{fig:heatmap} and \ref{fig:data}. First, in Figure \ref{fig:heatmap}, we consider the effects of an $\epsilon = 10^{-8}$ perturbation to the last pivot of Higham$^2$ matrices of dimension $n = 100$. Though $n = 100$ is quite a small test-case, it is already more than sufficient for our purposes, as, in double precision, the last pivot is nearly equal to the inverse of machine epsilon squared. The heatmap on the left is of a Higham$^2$ matrix generated by taking the $\hat U \in \mathrm{GL}_{n-1}(\mathbb{R})$ of Proposition \ref{prop} to have independent standard normal entries, scaled so that $\|A\|_{\max} = 1$, the map on the right is of the Wilkinson matrix (see Equation \ref{eqn:wilk}), and the map in the middle is of a matrix in between the two (w.r.t. choice of $\hat U$). Perturbations in the top left portion of a random Higham$^2$ matrix appear to be most impactful. This is consistent with Theorem \ref{theorem} and the fact that the inverse of an upper triangular matrix with normal entries tends to have exponentially large entries (in $n$) near the upper-right corner. However, the Wilkinson matrix provides a clear reminder that this is not always the case. The inverse of $\hat U = I$ has zero entries above the diagonal, rendering perturbations to the top-left entries of $A$ relatively useless. Our results are also consistent with Inequality \ref{ineq:top_right}: perturbing the top-right $(1,n-1)$ entry with a sufficiently large $\epsilon$ is a reliable way to always decrease the last pivot size.

\begin{figure}[t!]
    \centering
    \includegraphics[width=\textwidth]{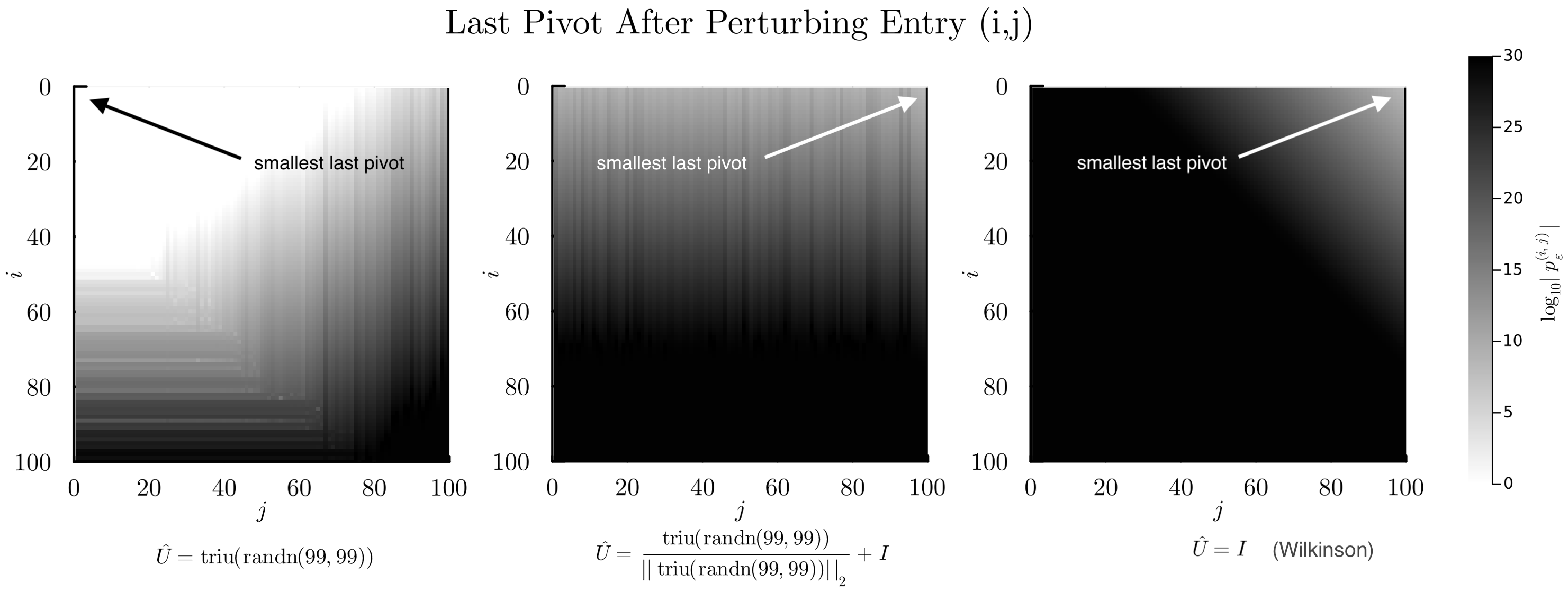}
    \caption{\small The effects of an $\epsilon = 10^{-8}$ perturbation at different entries of three $100 \times 100$ Higham$^2$ matrices for different choices of $\hat U$ ($\hat U = I$ is the Wilkinson matrix). We plot a heatmap for each, where the $(i, j)^{th}$ grid point is the value of $\log_{10}|p^{(i,j)}_\epsilon|$ (in exact arithmetic).}
       \label{fig:heatmap}
\end{figure}

\begin{figure}[t!]
    \centering
    \subfigure[$\hat U = \mathrm{triu}(\mathrm{randn}(n,n))$]{
        \includegraphics[width=0.315\textwidth]{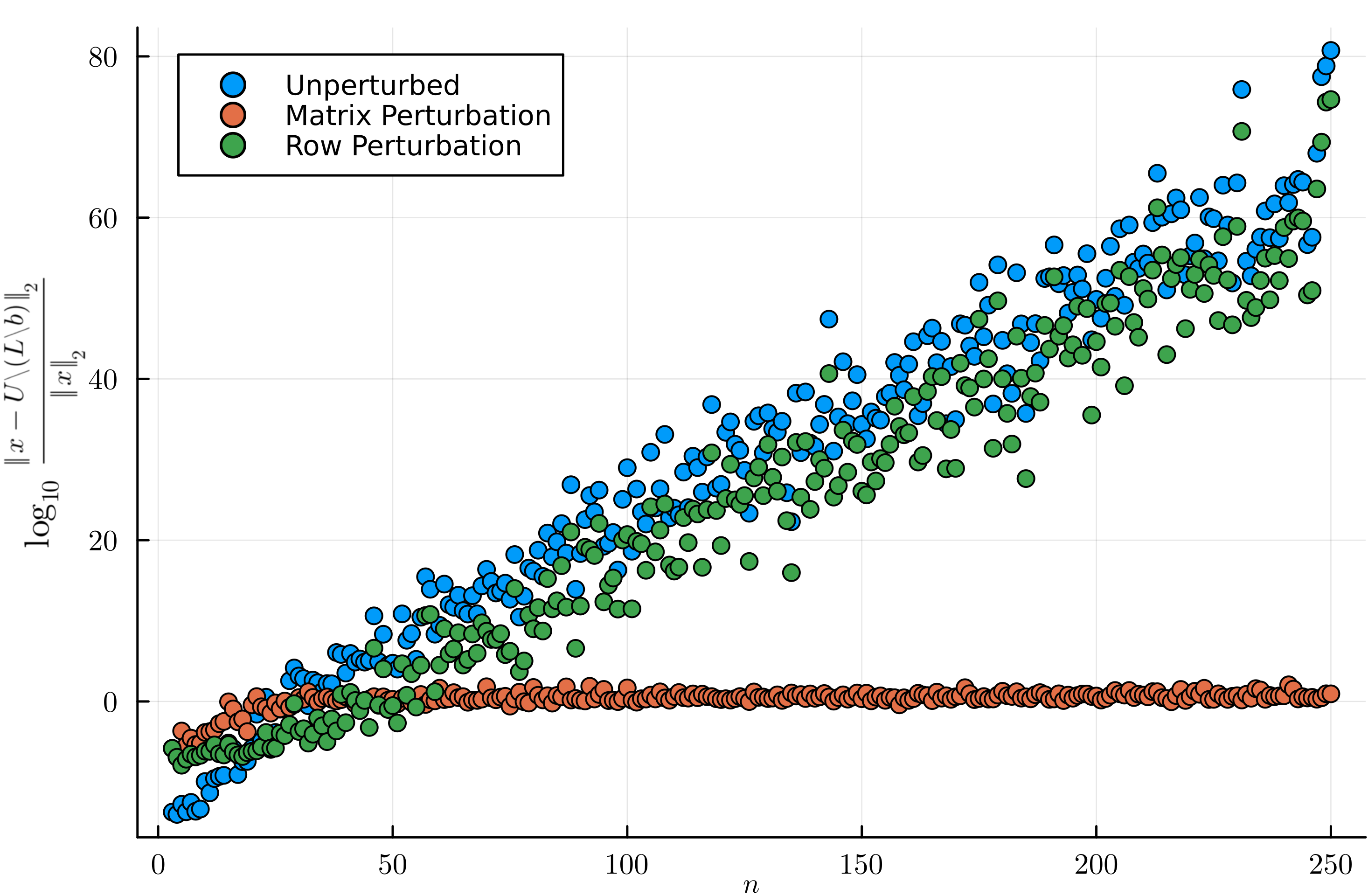}
    } \subfigure[$\hat U = \frac{\mathrm{triu}(\mathrm{randn}(n,n))}{\|\mathrm{triu}(\mathrm{randn}(n,n))\|_2}+I$]{\includegraphics[width=0.315\textwidth]{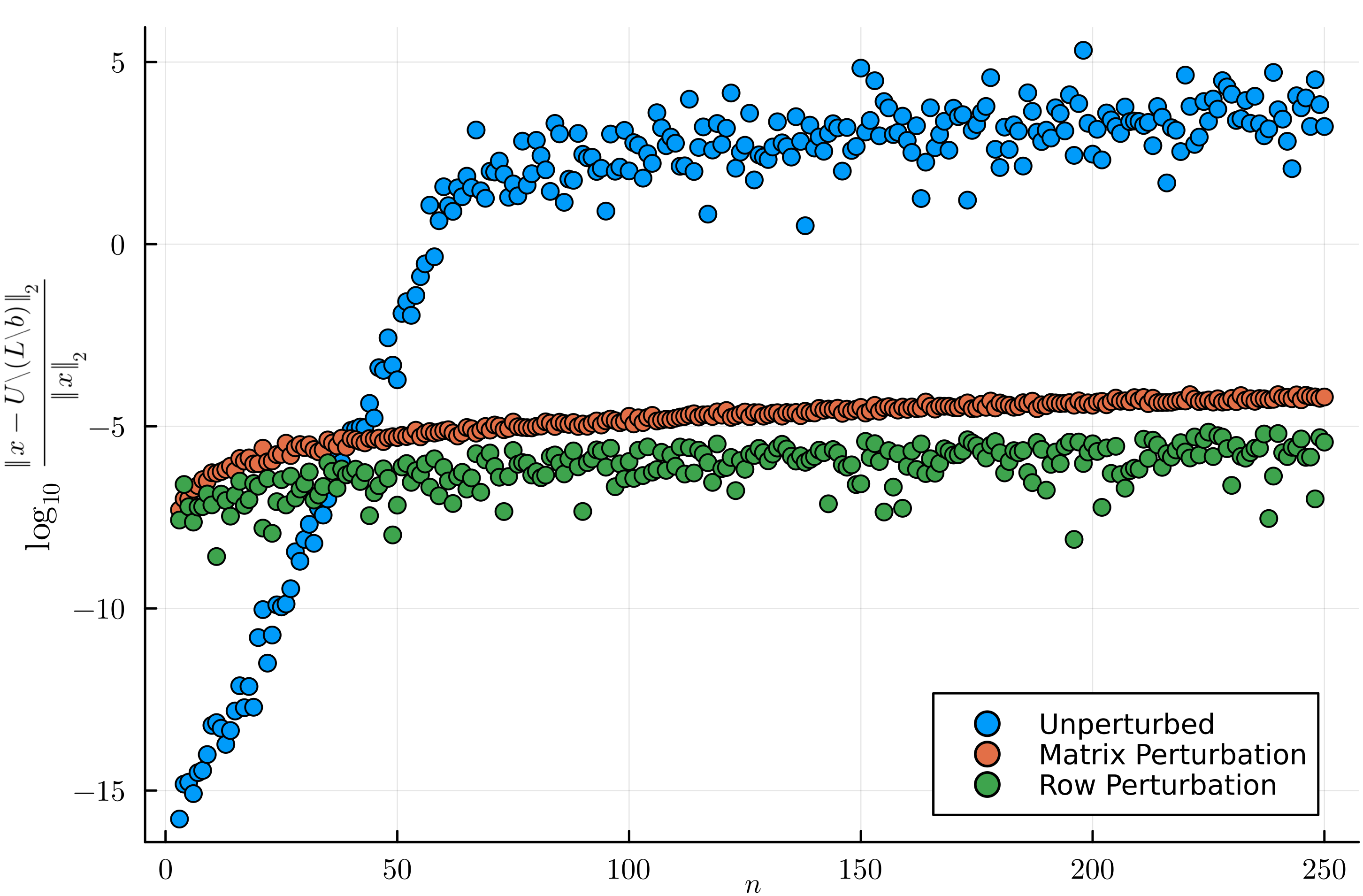}
    } \subfigure[$\hat U = I$ (Wilkinson)]{\includegraphics[width=0.315\textwidth]{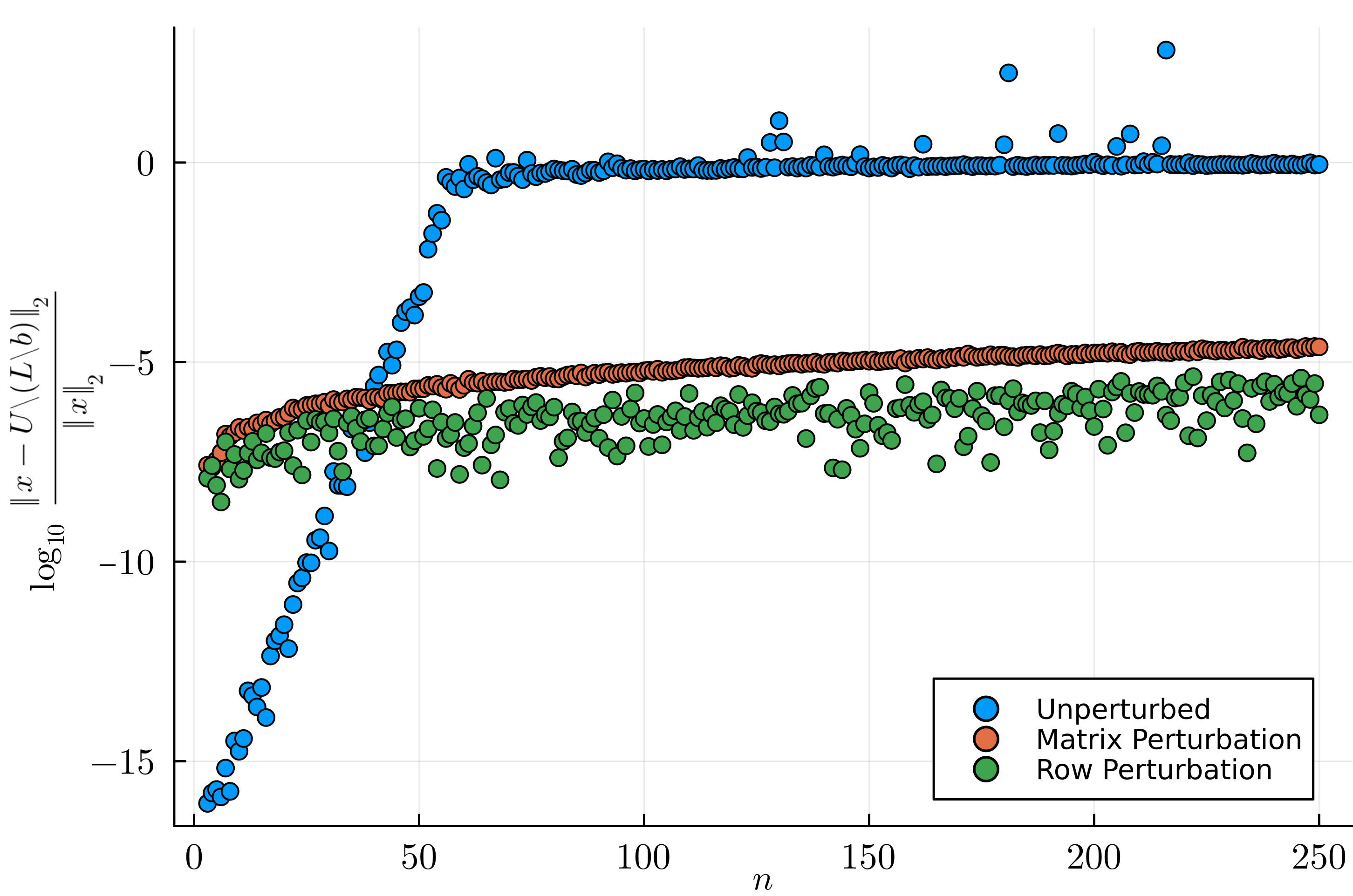}}
    \caption{\small Solving $Ax =b$ for Higham$^2$ matrices using Gaussian elimination with no pivoting (GENP) in double precision. For each value of $n$, we generate a Higham$^2$ matrix $A$ (in exact precision), a Gaussian vector $x$, compute $b = Ax$, and solve $A x =b$ for $x$ using GENP applied either to $A$ or $A+ 10^{-8} n B$, where $B$ has independent standard normal entries in either the first row only or the entire matrix. The three scatter plots report the log relative error for each choice of $\hat U$ and perturbation method.}
    \label{fig:data}
\end{figure}

    Of course, we are not merely interested in the last pivot, but in the quality of solution to $Ax = b$ we obtain using Gaussian elimination. Our theoretical and experimental results in Theorem \ref{theorem}, Corollary \ref{cor:topright}, and Figure \ref{fig:heatmap} give key insights into the stability of the last pivot, with implications for the growth factor itself, as every entry of $U$ is the last pivot of a sub-matrix of $A$. In Figure \ref{fig:data}, we examine the effects of matrix perturbation on the numerical solution to $Ax =b$ for Higham$^2$ matrices using Gaussian elimination with no pivoting. In the left plot, we observe that the ill-conditioning of a random Higham$^2$ matrix is a major barrier to a reasonable solution. This is consistent with the theoretical observation that extremely fast decay in the last pivot implies ill-conditioning (Corollary \ref{cor:cond}). In both the middle and right plots, we observe that the perturbation to the first row is the superior strategy. It is quite possible that a similar, albeit more complicated phenomenon holds more broadly than the class of maximum growth factor Higham$^2$ matrices considered here, and that there may be benefits to considering perturbations tailored specifically to Gaussian elimination with partial pivoting.

\section*{Acknowledgements}

\begin{figure}[b!]
    \centering
    \subfigure[A depiction of high dimensional $n\times n$ matrix space
portraying the Higham${}^2$ matrices as a blue curve and the
matrices with infinite growth as a red curve. The Wilkinson matrix (W) is shown as being far from the singular set. On the singular side of the blue curve, GE can be unstable.
]{ \includegraphics[width=0.55\textwidth]{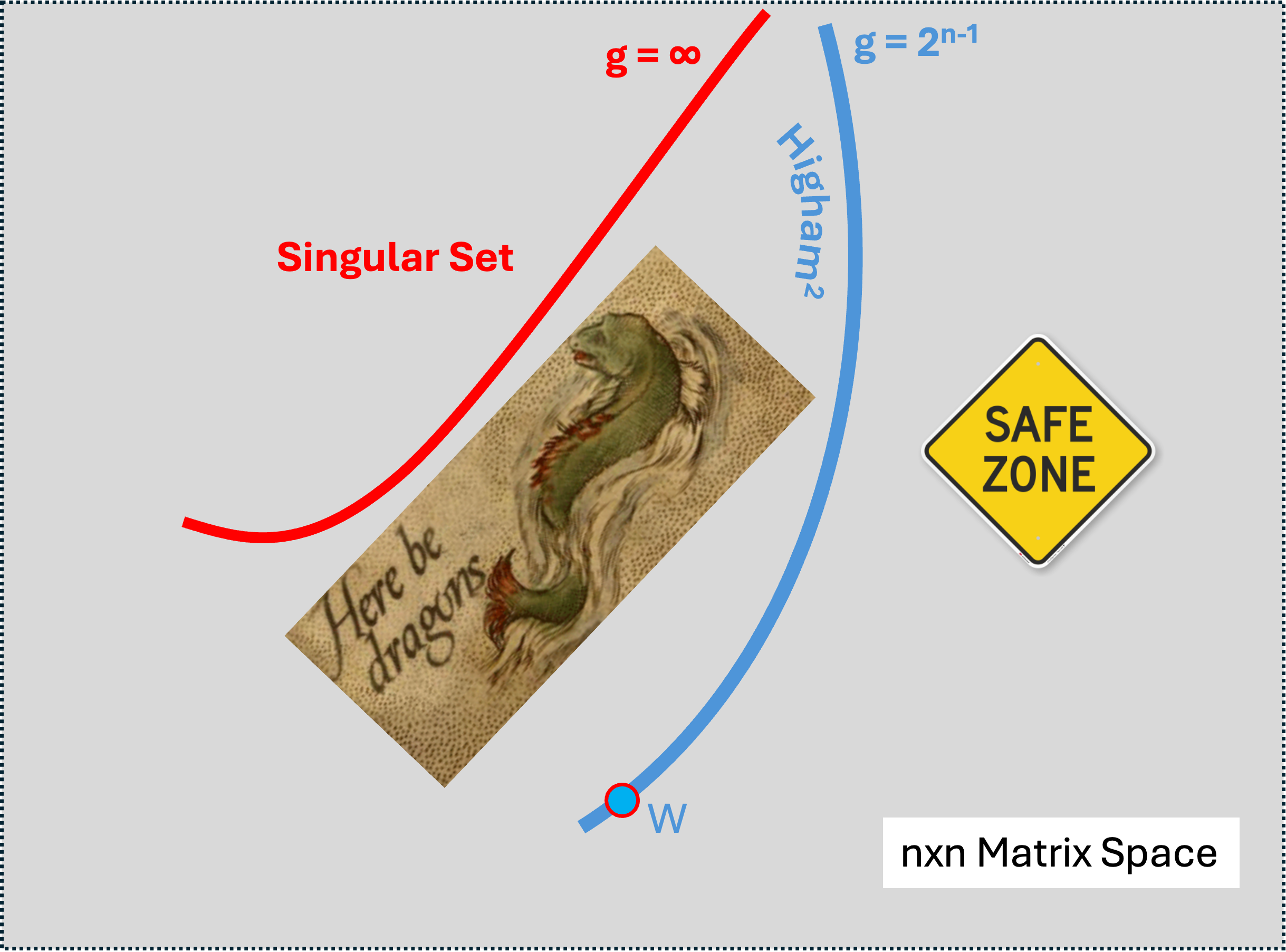}
    } \\ \subfigure[Matrix Space is in the plane of the floor with the growth factor g plotted in the z coordinate,
The Higham$^2$ matrices are portrayed as a straight line and the distance to singularity is true. Above the
Higham$^2$ matrices, a ``fence”  seems to appear as all have growth $2^{n-1}$.
]{\includegraphics[width=0.48\textwidth]{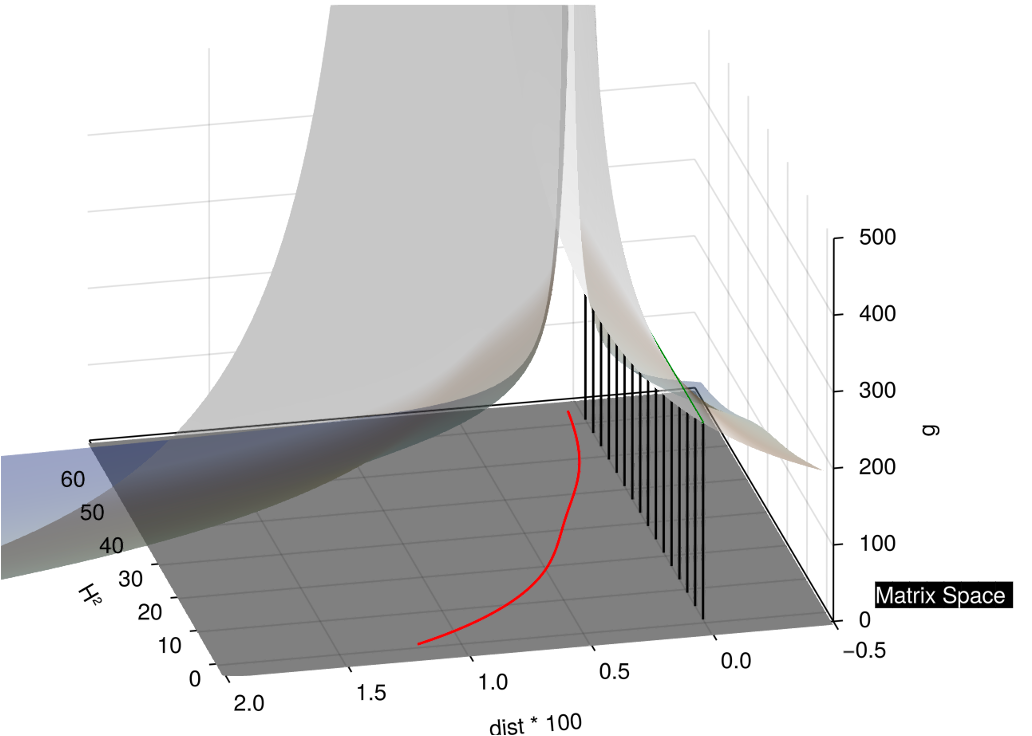}
    } \hspace{3 mm} \subfigure[In this figure we depict the
singular set as a straight line, so the
distance from Higham$^2$ matrices 
are distorted to be constant. Doing
so serves the interesting purpose of 
showing that g appears virtually
independent of the matrix in this 
distorted view.
]{\includegraphics[width=0.38\textwidth]{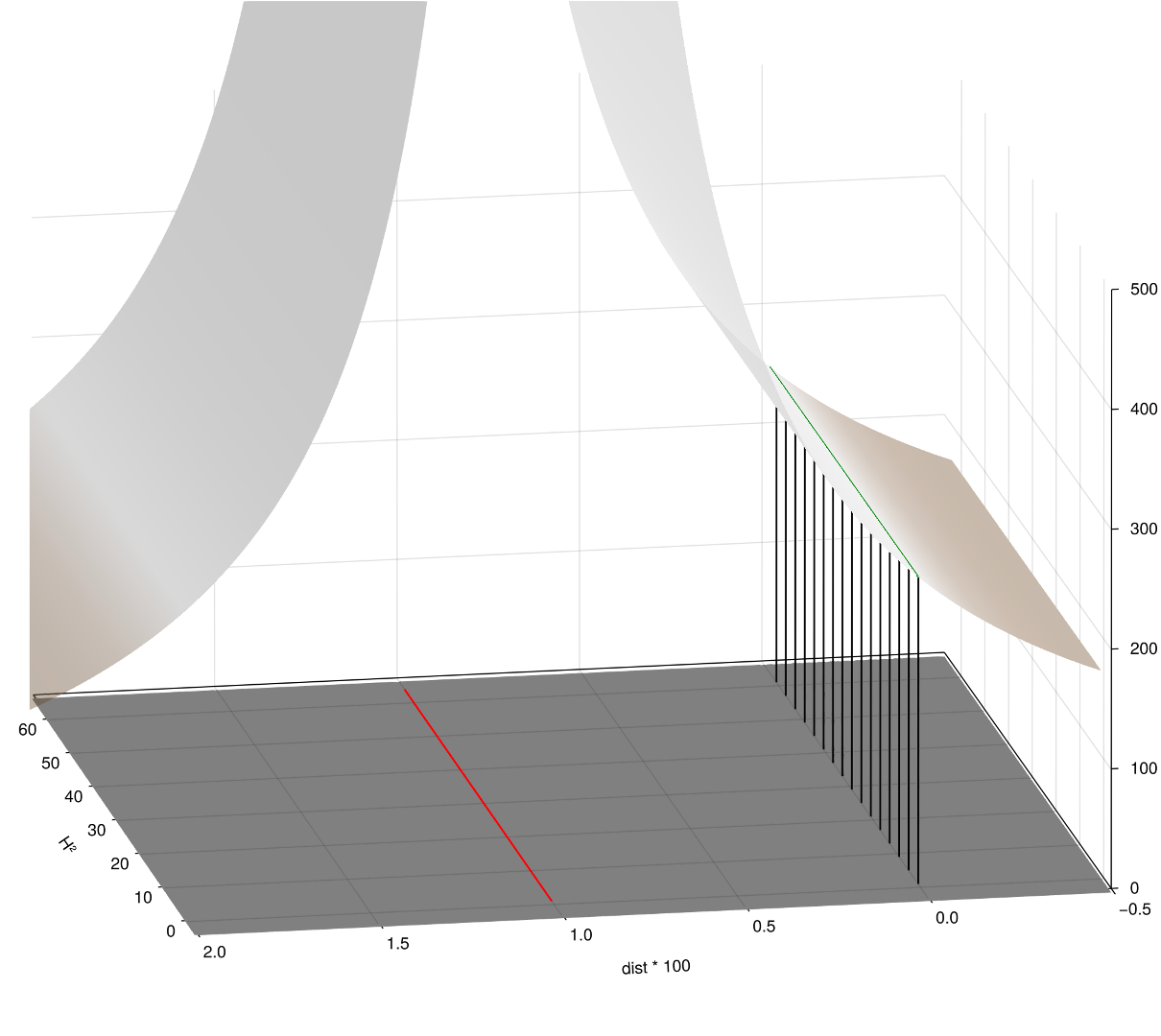}}
   \caption{}
\end{figure}

\setstretch{.4}
\begin{onehalfspacing}

{\tiny This material is based upon work supported by the National Science Foundation under grant no. OAC-1835443, grant no. SII-2029670, grant no. ECCS-2029670, grant no. OAC-2103804, and grant no. PHY-2021825. We also gratefully acknowledge the U.S. Agency for International Development through Penn State for grant no. S002283-USAID. The information, data, or work presented herein was funded in part by the Advanced Research Projects Agency-Energy (ARPA-E), U.S. Department of Energy, under Award Number DE-AR0001211 and DE-AR0001222. The views and opinions of authors expressed herein do not necessarily state or reflect those of the United States Government or any agency thereof. This material was supported by The Research Council of Norway and Equinor ASA through Research Council project ``308817 - Digital wells for optimal production and drainage''. Research was sponsored by the United States Air Force Research Laboratory and the United States Air Force Artificial Intelligence Accelerator and was accomplished under Cooperative Agreement Number FA8750-19-2-1000. The views and conclusions contained in this document are those of the authors and should not be interpreted as representing the official policies, either expressed or implied, of the United States Air Force or the U.S. Government. The U.S. Government is authorized to reproduce and distribute reprints for Government purposes notwithstanding any copyright notation herein. The third author's masters thesis \cite{bowen24} contains some of the results presented here, as well as some additional analysis and figures.}
\end{onehalfspacing}

{ \small 
	\bibliographystyle{plain}
	\bibliography{main.bib} }

 \end{document}